\documentclass[11pt]{amsart}
\usepackage{amsmath, amssymb, amsthm, amscd, graphicx}
\usepackage{epstopdf}
\usepackage{pst-all}

\theoremstyle{plain}
\newtheorem{prop}{Proposition}[section]
\newtheorem{coro}[prop]{Corollary}
\newtheorem{lemm}[prop]{Lemma}

\newtheorem{theorem}[prop]{Theorem}

\theoremstyle{definition}
\newtheorem{defi}[prop]{Definition}

\numberwithin{equation}{section}

\def\Reff#1; #2; #3; #4; #5; #6; #7\par{%
\bibitem{#1} #2, {\it #3}, #4 {\bf #5} (#6) #7}

\def\Ref#1; #2; #3; #4\par{%
\bibitem{#1} #2, {\it #3}, #4}

\renewcommand\aa[1]{a_{#1}}
\newcommand\aaa[2]{a_{#1, #2}}

\newcommand{\bb}[2]{b_{#1,#2}}
\newcommand\bbs[2]{b^*_{#1, #2}}
\newcommand\BKL[1]{B_{#1}^{{\scriptscriptstyle+}*}}
\newcommand\BP[1]{B_{#1}^{\scriptscriptstyle+}}
\newcommand\BR[1]{B_{#1}}

\newcommand\card{\mathtt{\#}}
\newcommand\Cat[1]{\mathrm{Cat}_{#1}}
\newcommand\CCCC{\mathbb{C}}

\newcommand\dd{d}
\newcommand\Deltad{\Delta^*}

\newcommand\ee{e}

\newcommand\ff{f}

\let\ge=\geqslant
\renewcommand{\gg}{g}
\newcommand\GG{G}

\newcommand\ii{i}
\newcounter{ITEM}
\newcommand\ITEM[1]{\setcounter{ITEM}{#1}\leavevmode\hbox{\rm(\roman{ITEM})}}
\newcommand\jj{j}

\newcommand{\inv}{^{-1}}

\newcommand{\kk}{k}
\newcommand{\kkl}{k-1}
\let\le=\leqslant
\newcommand\LGG[2]{\Vert#1\Vert_{#2}}

\newcommand\Mat[1]{A_{#1}^*}

\newcommand\NC[1]{\mathrm{NC}(#1)}
\newcommand\nn{n}
\newcommand\NN[2]{N^{#1}_{#2}}
\newcommand\NNp[2]{N^{#1}_{#2,+}}

\newcommand\pdots{\hspace{0.2ex}{\cdot}{\cdot}{\cdot}\hspace{0.2ex}}

\newcommand\PP{P}

\newcommand{\resp}{{\it resp.{~}}}

\newcommand\sig[1]{\sigma_{#1}}
\newcommand\siginv[1]{\sigma_{#1}^{-1}}
\renewcommand\ss{s}
\renewcommand\SS{S}
\newcommand\SSh{\widehat{S}}
\newcommand\Sym[1]{\mathfrak{S}_{#1}}

\renewcommand\tt{t}

\def\VR(#1,#2){\vrule width0pt height#1mm depth#2mm}

\newcommand\wdots{, ...\hspace{0.2ex},}

\newcommand{\zz}{z}

\newcommand\drawwhitedot[1]{\pscircle[linewidth=.8pt, fillstyle=solid](#1){2pt}}

\newpsstyle{zebra}{linecolor=red, fillstyle=vlines, hatchcolor=red, hatchwidth=0.5pt, hatchsep=1.5pt}

\def\Dv{%
\VR(6.5,4)\begin{picture}(10.5,0)(-4,-1)
\psset{unit=0.4mm, linewidth=.8pt, ncurv=1}\SpecialCoor\degrees[12]
\pnode(10;3){q1}\pnode(10;11){q2}\pnode(10;7){q3}
\pscircle[dimen=middle](0,0){10}
\drawwhitedot{q1}\drawwhitedot{q2}\drawwhitedot{q3}
\end{picture}}
\def\Da{%
\VR(6.5,4)\begin{picture}(10.5,0)(-4,-1)
\psset{unit=0.4mm, linewidth=.8pt, ncurv=1}\SpecialCoor\degrees[12]
\pnode(10;3){q1}\pnode(10;11){q2}\pnode(10;7){q3}
\pscircle[dimen=middle](0,0){10}
\psline(q1)(q2)
\drawwhitedot{q1}\drawwhitedot{q2}\drawwhitedot{q3}
\end{picture}}
\def\Db{%
\VR(6.5,4)\begin{picture}(10.5,0)(-4,-1)
\psset{unit=0.4mm, linewidth=.8pt, ncurv=1}\SpecialCoor\degrees[12]
\pnode(10;3){q1}\pnode(10;11){q2}\pnode(10;7){q3}
\pscircle[dimen=middle](0,0){10}
\psline(q2)(q3)
\drawwhitedot{q1}\drawwhitedot{q2}\drawwhitedot{q3}
\end{picture}}
\def\Dc{%
\VR(6.5,4)\begin{picture}(10.5,0)(-4,-1)
\psset{unit=0.4mm, linewidth=.8pt, ncurv=1}\SpecialCoor\degrees[12]
\pnode(10;3){q1}\pnode(10;11){q2}\pnode(10;7){q3}
\pscircle[dimen=middle](0,0){10}
\psline(q3)(q1)
\drawwhitedot{q1}\drawwhitedot{q2}\drawwhitedot{q3}
\end{picture}}
\def\DD{%
\VR(6.5,4)\begin{picture}(10.5,0)(-4,-1)
\psset{unit=0.4mm, linewidth=.8pt, ncurv=1}\SpecialCoor\degrees[12]
\pnode(10;3){q1}\pnode(10;11){q2}\pnode(10;7){q3}
\pscircle[dimen=middle](0,0){10}
\psline(q1)(q2)(q3)(q1)
\drawwhitedot{q1}\drawwhitedot{q2}\drawwhitedot{q3}
\end{picture}}

\begin{document}

\author{Philippe BIANE}
\address{Institut Gaspard-Monge, universit\'e Paris-Est Marne-la-Vall\'ee,
5 Boulevard Descartes, Champs-sur-Marne, 77454, Marne-la-Vall\'ee cedex 2,
France}
\email{biane@univ-mlv.fr}
\urladdr{monge.univ-mlv.fr/\textasciitilde biane/}

\author{Patrick DEHORNOY}
\address{Laboratoire de Math\'ematiques Nicolas Oresme UMR 6139, Universit\'e de Caen, 14032~Caen, France}
\email{dehornoy@math.unicaen.fr}
\urladdr{www.math.unicaen.fr/\textasciitilde dehornoy}

\title[Dual braids and free cumulants]{Dual Garside structure of braids and free cumulants of products}



\begin{abstract}
We count the $\nn$-strand braids whose normal decomposition has length at most~$2$ in the dual braid monoid~$\BKL\nn$ by reducing the question to a computation of free cumulants for a product of independent variables, for which we establish a general formula.
\end{abstract}

\maketitle

\section{Introduction}

A Garside structure on a group~$\GG$ is a generating family~$\SS$ that gives rise to distinguished decompositions of a particular type, leading in good cases to an automatic structure on~$\GG$ and, from there, to solutions of the word and conjugacy problems of~$\GG$~\cite{Garside}. In the case of the $\nn$-strand Artin braid group~$\BR\nn$, two Garside structures are known: the so-called classical Garside structure, in which the distinguished generating family is a copy of the symmetric group~$\mathfrak{S}_\nn$~\cite[Chapter~9]{Eps}, and the so-called \emph{dual Garside structure}, in which the distinguished generating family is a copy of the family~$\NC\nn$ of all size~$\nn$ noncrossing partitions~\cite{BDM}.

Whenever a finite Garside structure~$\SS$ is given on a group~$\GG$, natural counting problems arise, namely counting how many elements of the group~$\GG$ or of the submonoid generated by~$\SS$ have length (at most)~$\ell$ with respect to~$\SS$. Call $\SS$-normal the distinguished decompositions associated with a Garside structure~$\SS$. As $\SS$-normal sequences happen to be geodesic, the above question amounts to counting $\SS$-normal sequences of length~$\ell$. And, because $\SS$-normality is a purely local property, the central question is to determine $\SS$-normal sequences of length~two, the general case then corresponding to taking the $\ell$th power of the incidence matrix associated with length two.

Initially motivated by the investigation of the logical strength of certain statements involving the standard braid ordering~\cite{Dhq}, the above mentioned counting questions in the case of the classical Garside structure of braids have been addressed in~\cite{Dhi}, leading to nontrivial results involving Solomon's descent algebra and to natural conjectures, like the one established in~\cite{HNT} using the theory of quasi-symmetric functions.

The aim of this paper is to address similar questions in the case of the dual Garside structure of braids and to obtain an explicit determination of the generating function for the number of normal sequences of length two  (here $\Cat\nn$ denotes the $\nn$th Catalan number):

\begin{prop}
\label{P:Dual}
Let $\bbs\nn2$ be the number of braids of length at most~$2$ in the dual braid monoid~$\BKL\nn$. Then the function~$R(\zz) = 1 + \sum_{\nn\ge1} \bbs\nn2 \zz^\nn$ is connected with $M(\zz) = 1 + \sum_{\nn\ge1} \Cat\nn^2 \zz^\nn$ by the equality
\begin{equation}
\label{E:Dual}
R(\zz M(\zz)) = M(\zz).
\end{equation}
\end{prop}

This formula, which inductively determines the numbers~$\bbs\nn2$, will be deduced from a general formula for computing the free cumulants of a product of independent random variables.

Free cumulants where invented by Roland Speicher in order to compute with free random variables~\cite{NiSp}. In particular, free cumulants give a simple way of computing the free additive convolution of two probability measures on the real line, namely the free cumulants of the sum of two free random variables are the sums of the free cumulants of the variables. Free cumulants also appear when one tries to compute the free cumulants of a product of free random variables, although in a more complicated way. 

Here is the general formula we establish (the notations are standard and explained in section~\ref{S:Cumulants}):

\begin{theorem}
\label{T:CumProd}
Let $X_1\wdots X_k$ be a family of commuting independent random variables, and let $R^{(i)}_n$ be the free cumulants of $X_i$. Then the free cumulants of the product $X_1X_2\pdots X_k$ are given by\begin{equation}
R_n=\sum_{\pi_1\vee\pdots\vee\pi_k={\bf 1}_n}\prod_iR^{(i)}_{\pi_i},
\end{equation}
the sum being over all $k$-tuples of noncrosssing partitions in $\NC\nn$ whose join is the largest partition.
\end{theorem}

Free cumulants are specially designed to deal with highly noncommuting objects, therefore it may come as a surprise that there is also a simple formula for computing free cumulants of a product of independent commuting random variables, in terms of the free cumulants of the factors. Actually the formula also holds, with appropriate modifications, for classical and Boolean cumulants.

The organization of the paper is as follows. In Section~\ref{S:Dual}, we recall the description of the dual Garside structure of braid groups and raise the induced counting questions. In Section~\ref{S:Cumulants}, we review basic definitions about free cumulants and establish Theorem~\ref{T:CumProd}. This part can be read independently of the rest of the paper. In Section~\ref{S:Back}, we apply the result of Section~\ref{S:Cumulants} to braids and conclude with further questions and one additional result about the determinant of the involved incidence matrix.

\section{The dual Garside structure of braids}
\label{S:Dual}

\subsection*{Braid groups}

For $\nn \ge 1$, the $\nn$-strand braid group~$B_\nn$ is the group defined by the presentation
\begin{equation} 
\label{E:Present}
B_\nn = \left<\sig1 \wdots  \sig{\nn-1} \,\bigg\vert\,
\begin{array}{cl}
\sig i \sig j = \sig j \sig i 
&\text{\quad for $\vert i - j\vert \ge 2$}\\ \   \sig i \sig j \sig i  = \sig j \sig i \sig j
&\text{\quad for $\vert i - j\vert = 1$}
\end{array} \right>.
\end{equation}
The group~$\BR\nn$ is both the group of isotopy classes of $\nn$-strand geometric braids, the mapping class group of an $\nn$-punctured disk, and the fundamental group of the configuration space obtained by letting the symmetric group act on the complement of the diagonal hyperplanes in~$\CCCC^\nn$~\cite{Bir}.

\subsection*{Garside structures}

A \emph{Garside structure} in a group~$\GG$ is a subset~$\SS$ of~$\GG$ such that every element of~$\GG$ admits a decomposition of a certain syntactic form in terms of the elements of~$\SS$, namely a symmetric $\SS$-normal decomposition in the following sense.

\begin{defi}\cite{Garside}
Assume that $\GG$ is a group and $\SS$ is included in~$\GG$. 

\ITEM1 A finite $\SS$-sequence $(\ss_1 \wdots \ss_\dd)$ is called \emph{$\SS$-normal} if, for~$\ii < \dd$, every element of~$\SS$ left-dividing $\ss_{\ii}\ss_{\ii+1}$ left-divides~$\ss_{\ii}$, where ``$\ff$ left-divides~$\gg$'' means ``$\ff\inv \gg$ lies in the submonoid~$\widehat\SS$ of~$\GG$ generated by~$\SS$''.

\ITEM2 A pair of finite $\SS$-sequences $((\ss_1 \wdots \ss_\dd), (\tt_1 \wdots \tt_\ee))$ is called \emph{symmetric $\SS$-normal} if $(\ss_1 \wdots \ss_\dd)$ and $(\tt_1 \wdots \tt_\ee)$ are $\SS$-normal and, in addition, the only element of~$\SS$ left-dividing~$\ss_1$ and $\tt_1$ is~$1$.

\ITEM3 We say that $(\ss_1 \wdots \ss_\dd)$ (\resp $((\ss_1 \wdots \ss_\dd), (\tt_1 \wdots \tt_\ee))$) is a \emph{decomposition} for an element~$\gg$ of~$\GG$ if $\gg = \ss_1 \pdots \ss_\dd$ (\resp $\gg = \tt_{\ee}\inv \pdots \tt_1\inv \ss_1 \pdots \ss_{\dd}$) holds in~$\GG$.
\end{defi}

Every group~$\GG$ is trivially a Garside structure in itself, and the notion is interesting only when $\SS$ is small, typically when $\GG$ is infinite and $\SS$ is finite or, at least, is properly included in~$\GG$. Under mild assumptions, the existence of a finite Garside structure implies good properties for the group~$\GG$ such as the existence of an automatic structure or the decidability of the word and conjugacy problems.

Whenever $\SS$ is a finite generating family in a group~$\GG$, it is natural to consider the numbers
$$\NN\SS\dd = \card\{\gg \in \GG \mid \LGG\gg\SS = \dd\},$$
where $\LGG\gg\SS$ refers to the minimal length of an $\SS$-decomposition.
In the case of a Garside structure, symmetric $\SS$-normal decompositions are (essentially) unique, and they are geodesic. Therefore, $\NN\SS\dd$ identifies with the number of length~$\dd$ symmetric $\SS$-normal sequences. In such a context, the submonoid~$\SSh$ of~$\GG$ generated by~$\SS$ is the family of all elements of~$\GG$ whose symmetric $\SS$-normal decomposition has an empty denominator, that is, of all elements that admit an $\SS$-normal decomposition. Then, it is also natural to introduce the number
$$\NNp\SS\dd = \card\{\gg \in \SSh \mid \LGG\gg\SS = \dd\},$$
which, by the above remark, is the number of $\SS$-normal sequences of length~$\dd$. It then follows from the definition that a sequence $(\ss_1 \wdots \ss_\dd)$ is $\SS$-normal if and only if every length two subsequence is $\SS$-normal, and the basic question is therefore to investigate the numbers
\begin{equation}
\label{E:General}
\NNp\SS2 = \card\{(\ss_1, \ss_2) \in \SS^2 \mid (\ss_1, \ss_2) \mbox{ is $\SS$-normal}\}.
\end{equation}

\subsection*{The classical Garside structure of~$\BR\nn$}

In the case of the braid group~$\BR\nn$, two Garside structures are known. The first one, often called \emph{classical}, involves permutations. By~\eqref{E:Present}, mapping~$\sig\ii$ to the transposition $(\ii, \ii{+}1)$ induces a surjective homomorphism~$\mathrm{pr}_\nn: \BR\nn \to \Sym\nn$. Considering the positive braid diagrams in which any two strands cross at most once (``simple braids'') provides a set-theoretic section for~$\mathrm{pr}_\nn$, whose image is a copy~$\SS_\nn$ of~$\Sym\nn$ inside~$\BR\nn$. The family~$\SS_\nn$ is a Garside family in~$\BR\nn$~\cite[Chapter~9]{Eps}, the submonoid of~$\BR\nn$ generated by~$\SS_\nn$ being the submonoid~$\BP\nn$ generated by~$\sig1 \wdots \sig{\nn-1}$~\cite{Gar}. The associated numbers~$\NNp{\SS_\nn}\dd$ have been investigated in~\cite{Dhi}. In particular, writing~$\bb\nn\dd$ for~$\NNp{\SS_\nn}\dd$, it is shown that the numbers~$\bb\nn2$ are determined by the induction
\begin{equation}
\bb\nn2 = \sum_{\ii=0}^{\nn-1} (-1)^{\nn + \ii +1} {\nn \choose \ii}^2 \bb\ii2
\end{equation}
and that the double exponential series $\sum \bb\nn2 \zz^\nn/(\nn!)^2$ is the inverse of the Bessel function~$J_0(\sqrt\zz)$.

\subsection*{The dual Garside structure of~$\BR\nn$}

It is known since~\cite{BKL} that, for every~$\nn$, there exists an alternative Garside structure on~$B_\nn$, namely the \emph{dual} Garside structure~$\SS_\nn^*$, whose elements are in one-to-one correspondence with the noncrossing partitions of~$\{1, 2 \wdots \nn\}$. The question then naturally arises of determining the derived numbers $\NN{\SS_\nn^*}\dd$ and $\NNp{\SS_\nn^*}\dd$. Here we shall concentrate on the latter, hereafter denoted~$\bbs\nn\dd$, and specifically on~$\bbs\nn2$. By~\eqref{E:General}, we have
\begin{equation}
\label{E:Dual}
\bbs\nn\dd = \card\{(\ss_1 \wdots \ss_\dd) \in (\SS_\nn^*)^\dd \mid \forall\ii< \nn \, ((\ss_\ii , \ss_{\ii+1}) \mbox{ is $\SS_\nn^*$-normal)}\}.
\end{equation}

In order to compute the numbers~$\bbs\nn\dd$, we shall describe the correspondence between~$\SS_\nn^*$ and noncrossing partitions and interpret the $\SS_\nn^*$-normality condition in terms of the latter.

We recall that a set partition of the set $\{1 \wdots n\}$ is called noncrossing if there is no quadruple $1\le i<j<k<l\le n$ such that $i$ and $k$ belong to some part of the partition, whereas $j$ and $l$ belong to another part. The set of noncrossing partitions of $\{1 \wdots n\}$, denoted by $\NC\nn$, is a poset for the reverse refinement order: for two  partitions $\pi$ and $\pi'$ one has $\pi\le \pi'$ if and only if each part of $\pi$ is included in some part of $\pi'$. With this order, $\NC\nn$ is a lattice, with largest element ${\bf 1}_n$ the partition with only one part, and smallest element ${\bf 0}_n$, the partition into $n$ singletons. We will denote by $\pi_1 \vee \pdots\vee\pi_\dd$ the join of a family of noncrossing partitions.

\begin{defi}\cite{BKL}
\label{D:FRGen}
For $1 \le \ii < \jj$, put $\aaa\ii\jj = \sig\ii ... \sig{\jj-2} \, \sig{\jj-1} \, \siginv{\jj-2} ... \siginv\ii$ in~$\BR\nn$. The \emph{dual braid monoid}~$\BKL\nn$ is the submonoid of~$\BR\nn$ generated by all elements~$\aaa\ii\jj$, and $\SS_\nn^*$ is the family of all left-divisors of~$\Deltad_\nn$ in~$\BKL\nn$, with $\Deltad_\nn = \sig1\sig2 \pdots \sig{\nn-1}$.
\end{defi}

Note that $\sig\ii = \aaa\ii{\ii+1}$ holds for every~$\ii$, hence $\BKL\nn$ includes~$\BP\nn$, a proper inclusion for $\nn \ge 3$. 

\begin{prop}\cite{BKL}
The family~$\SS^*_\nn$ is a Garside structure on the group~$\BR\nn$.
\end{prop}

\noindent\begin{minipage}{\textwidth}
\rightskip40mm\VR(4,0)\hspace{8pt} 
It is convenient to associate with the $\nn$-strand braid~$\aaa\ii\jj$ a graphical representation as the chord~$(\ii, \jj)$ in a disk with $\nn$~marked vertices on the border as shown on the right.
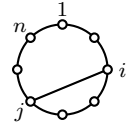
\hfill\begin{picture}(0,0)(-20,-8)
\psset{unit=0.6mm, linewidth=.8pt, ncurv=1}\SpecialCoor\degrees[8]
\pnode(10;2){q1}\pnode(10;1){q2}\pnode(10;0){q3}\pnode(10;7){q4}
\pnode(10;6){q5}\pnode(10;5){q6}\pnode(10;4){q7}\pnode(10;3){q8}
\pscircle[dimen=middle](0,0){10}
\psline[fillstyle=solid,fillcolor=gray](q3)(q6)
\drawwhitedot{q1}\drawwhitedot{q2}\drawwhitedot{q3}\drawwhitedot{q4}
\drawwhitedot{q5}\drawwhitedot{q6}\drawwhitedot{q7}\drawwhitedot{q8}
\psset{unit=0.8mm}
\pnode(10;2){q1}\pnode(10;1){q2}\pnode(10;0){q3}\pnode(10;7){q4}
\pnode(10;6){q5}\pnode(10;5){q6}\pnode(10;4){q7}\pnode(10;3){q8}
\rput(q1){$\scriptstyle1$}\rput(q3){$\scriptstyle\ii$}\rput(q6){$\scriptstyle\jj$}\rput(q8){$\scriptstyle\nn$}
\end{picture}\\
\end{minipage}

Then the correspondence between the elements of~$\SS_\nn^*$ and noncrossing partitions stems from the following observation:

\begin{lemm}\cite{BDM}
\label{L:Order}
For $\PP$ a union of disjoint polygons in the $\nn$-marked disk, say $\PP = \PP_1 \cup \pdots \cup \PP_\dd$, let $\aa\PP = \aa{\PP_1} \pdots \aa{\PP_\dd}$, with $\aa{\PP_\kk} = \aaa{\ii_1}{\ii_2} \aaa{\ii_2}{\ii_3} \pdots \aaa{\ii_{\nn_\kk-1}}{\ii_{\nn_\kk}}$ where $(\ii_1 \wdots \ii_{\nn_\kk})$ is a clockwise enumeration of the vertices of~$\PP_\kk$. 

\ITEM1 The braid~$\aa\PP$ only depends on~$\PP$ and not on the order of enumeration.

\ITEM2 Mapping~$\PP$ to~$\aa\PP$ establishes a bijection between unions of disjoint polygons in the $\nn$-marked disk and elements of~$\SS_\nn^*$.
\end{lemm}

It is standard to define a bijection between noncrossing partitions of~$\{1 \wdots \nn\}$ and unions of disjoint polygons in the $\nn$-marked disk, yielding the announced correspondence.  Noncrossing partitions can be embedded into the symmetric group, using geodesics in the Cayley graph \cite{Bia}: this amounts to mapping a union of polygons to the product of the cycles obtained by enumerating their vertices in clockwise order. Note that, although noncrossing partitions may be viewed as particular permutations, the associated braids need not coincide: for instance, the cycle~$(1 3)$, which corresponds to the partition $\{\{1, 3\}, \{2\}\}$, is associated in~$\SS_3^*$ with the braid~$\aaa13$, that is, $\sig1\sig2\siginv1$, whereas it associated with~$\sig1\sig2\sig1$ in~$\SS_3$.

\noindent\begin{minipage}{\textwidth}
\rightskip40mm\VR(4,0)\hspace{8pt} 
For~$\pi$ a noncrossing partition of~$\{1 \wdots \nn\}$, we shall denote by~$\aa\pi$ the associated element of~$\SS_\nn^*$. For instance, for $\pi = \{\{1\}, \{2,8\}, \{3,5,6\}, \{4\}, \{7\}\}$, we find $\aa\pi = \aaa28 \aaa35 \aaa56$, as shown on the right.
\hfill\begin{picture}(0,0)(-20,-8)
\psset{unit=0.6mm, linewidth=.8pt, ncurv=1}\SpecialCoor\degrees[8]
\pnode(10;2){q1}\pnode(10;1){q2}\pnode(10;0){q3}\pnode(10;7){q4}
\pnode(10;6){q5}\pnode(10;5){q6}\pnode(10;4){q7}\pnode(10;3){q8}
\pscircle[dimen=middle](0,0){10}
\psline(q3)(q5)(q6)(q3)
\psline(q2)(q8)
\drawwhitedot{q1}\drawwhitedot{q2}\drawwhitedot{q3}\drawwhitedot{q4}
\drawwhitedot{q5}\drawwhitedot{q6}\drawwhitedot{q7}\drawwhitedot{q8}
\psset{unit=0.8mm}
\pnode(10;2){q1}\pnode(10;1){q2}\pnode(10;0){q3}\pnode(10;7){q4}
\pnode(10;6){q5}\pnode(10;5){q6}\pnode(10;4){q7}\pnode(10;3){q8}
\rput(q1){$\scriptstyle1$}\rput(q2){$\scriptstyle2$}\rput(q3){$\scriptstyle3$}\rput(q4){$\scriptstyle4$}\rput(q5){$\scriptstyle5$}\rput(q6){$\scriptstyle6$}\rput(q7){$\scriptstyle7$}\rput(q8){$\scriptstyle8$}
\end{picture}\\
\end{minipage}

Under the above correspondence, $\Deltad_\nn$ corresponds to the (unique) $\nn$-gon in the $\nn$-marked disk, hence to the (noncrossing) partition~${\bf1}_\nn$ with one part. Owing to~\eqref{E:Dual}, the numbers~$\bbs\nn\dd$ are determined by
\begin{equation}
\label{E:Dual2}
\bbs\nn\dd = \card\{(\pi_1 \wdots \pi_\dd) \in \NC\nn^\dd \mid \forall\ii< \nn \, ((\aa{\pi_\ii}, \aa{\pi_{\ii+1}}) \mbox{ is $\SS_\nn^*$-normal})\},
\end{equation}
and we have to recognize when the braids associated with two partitions make an $\SS_\nn^*$-normal sequence.

By construction, the Garside structure~$\SS_\nn^*$ is what is called bounded by the element~$\Deltad_\nn$, that is, it exactly consists of the left-divisors of~$\Deltad_\nn$ in~$\BKL\nn$. In this case, the normality condition takes a simple form.

\begin{lemm}\cite[Chapter~VI]{Garside}
Assume that $\SS$ is a Garside structure that is bounded by an element~$\Delta$ in a group~$\GG$. Then, for~$\ss, \tt$ in~$\SS$, the pair~$(\ss, \tt)$ is $\SS$-normal if and only if the only common left-divisor of~$\partial\ss$ and~$\tt$ in the monoid~$\SSh$ is~$1$, where $\partial\ss$ is the element satisfying $\ss \cdot \partial\ss = \Delta$.
\end{lemm}

We are thus left with the question of recognizing in terms of noncrossing partitions (or of the corresponding unions of polygons) when a braid~$\aaa\ii\jj$ left-divides the braid~$\aa\pi$, and what is the partition~$\pi'$ satisfying $\aa{\pi'} = \partial\aa\pi$.

\begin{lemm}\cite{BDM}
\label{L:Complement}
Assume that $\pi$ lies in~$\NC\nn$. 

\ITEM1 For $1 \le \ii < \jj \le \nn$, the braid~$\aaa\ii\jj$ left-divides~$\aa\pi$ in~$\BKL\nn$ if and only if the chord~$(\ii, \jj)$ lies inside the union of the convex hulls of the polygons associated with~$\pi$.

\ITEM2 We have $\partial\aa\pi = \aa{\overline\pi}$, where $\overline\pi$ is the Kreweras complement of~$\pi$.
\end{lemm}

\begin{figure}[htb]
\begin{picture}(25,27)(-10,-13)
\psset{unit=1.1mm, linewidth=.8pt, ncurv=1}\SpecialCoor\degrees[12]
\pnode(10;3){q1}\pnode(10;2){q2}\pnode(10;1){q3}\pnode(10;0){q4}
\pnode(10;11){q5}\pnode(10;10){q6}\pnode(10;9){q7}\pnode(10;8){q8}
\pnode(10;7){q9}\pnode(10;6){q10}\pnode(10;5){q11}\pnode(10;4){q12}
\pscircle[dimen=middle](0,0){10}
\psline[fillstyle=solid,fillcolor=gray](q1)(q5)(q12)(q1)
\psline[fillstyle=solid,fillcolor=gray](q2)(q3)
\psline[fillstyle=solid,fillcolor=gray](q6)(q8)(q9)(q6)
\drawwhitedot{q1}\drawwhitedot{q2}\drawwhitedot{q3}\drawwhitedot{q4}
\drawwhitedot{q5}\drawwhitedot{q6}\drawwhitedot{q7}\drawwhitedot{q8}
\drawwhitedot{q9}\drawwhitedot{q10}\drawwhitedot{q11}\drawwhitedot{q12}
\psset{unit=1.4mm}
\pnode(10;3){q1}\pnode(10;2){q2}\pnode(10;1){q3}\pnode(10;0){q4}
\pnode(10;11){q5}\pnode(10;10){q6}\pnode(10;9){q7}\pnode(10;8){q8}
\pnode(10;7){q9}\pnode(10;6){q10}\pnode(10;5){q11}\pnode(10;4){q12}
\rput(q1){$\scriptstyle1$}\rput(q2){$\scriptstyle2$}\rput(q3){$\scriptstyle3$}\rput(q4){$\scriptstyle4$}\rput(q5){$\scriptstyle5$}\rput(q6){$\scriptstyle6$}\rput(q7){$\scriptstyle7$}\rput(q8){$\scriptstyle8$}\rput(q9){$\scriptstyle9$}\rput(q10){$\scriptstyle10$}\rput(q11){$\scriptstyle11$}\rput(q12){$\scriptstyle12$}
\end{picture}
\hspace{10mm}
\begin{picture}(25,27)(-10,-13)
\psset{unit=1.1mm, linewidth=.8pt, ncurv=1}\SpecialCoor\degrees[12]
\pnode(10;3){q1}\pnode(10;2){q2}\pnode(10;1){q3}\pnode(10;0){q4}
\pnode(10;11){q5}\pnode(10;10){q6}\pnode(10;9){q7}\pnode(10;8){q8}
\pnode(10;7){q9}\pnode(10;6){q10}\pnode(10;5){q11}\pnode(10;4){q12}
\pscircle[dimen=middle](0,0){10}
\psline[linestyle=dotted, dotsep=1.5pt](q1)(q5)(q12)(q1)
\psline[linestyle=dotted, dotsep=1.5pt](q2)(q3)
\psline[linestyle=dotted, dotsep=1.5pt](q6)(q8)(q9)(q6)
\psline[linecolor=red](q1)(q3)(q4)(q1)
\psline[linecolor=red](q5)(q9)(q10)(q11)(q5)
\psline[linecolor=red](q6)(q7)
\drawwhitedot{q1}\drawwhitedot{q2}\drawwhitedot{q3}\drawwhitedot{q4}
\drawwhitedot{q5}\drawwhitedot{q6}\drawwhitedot{q7}\drawwhitedot{q8}
\drawwhitedot{q9}\drawwhitedot{q10}\drawwhitedot{q11}\drawwhitedot{q12}
\psset{unit=1.4mm}
\pnode(10;3){q1}\pnode(10;2){q2}\pnode(10;1){q3}\pnode(10;0){q4}
\pnode(10;11){q5}\pnode(10;10){q6}\pnode(10;9){q7}\pnode(10;8){q8}
\pnode(10;7){q9}\pnode(10;6){q10}\pnode(10;5){q11}\pnode(10;4){q12}
\rput(q1){$\scriptstyle1$}\rput(q2){$\scriptstyle2$}\rput(q3){$\scriptstyle3$}\rput(q4){$\scriptstyle4$}\rput(q5){$\scriptstyle5$}\rput(q6){$\scriptstyle6$}\rput(q7){$\scriptstyle7$}\rput(q8){$\scriptstyle8$}\rput(q9){$\scriptstyle9$}\rput(q10){$\scriptstyle10$}\rput(q11){$\scriptstyle11$}\rput(q12){$\scriptstyle12$}
\end{picture}
\hspace{10mm}
\begin{picture}(25,27)(-10,-13)
\psset{unit=1.1mm, linewidth=.8pt, ncurv=1}\SpecialCoor\degrees[12]
\pnode(10;3){q1}\pnode(10;2){q2}\pnode(10;1){q3}\pnode(10;0){q4}
\pnode(10;11){q5}\pnode(10;10){q6}\pnode(10;9){q7}\pnode(10;8){q8}
\pnode(10;7){q9}\pnode(10;6){q10}\pnode(10;5){q11}\pnode(10;4){q12}
\pscircle[dimen=middle](0,0){10}
\psline[fillstyle=solid,fillcolor=gray](q1)(q2)
\psline[fillstyle=solid,fillcolor=gray](q4)(q11)(q12)(q4)
\psline[fillstyle=solid,fillcolor=gray](q5)(q7)(q8)(q5)
\psline[linecolor=red](q1)(q3)(q4)(q1)
\psline[linecolor=red](q5)(q9)(q10)(q11)(q5)
\psline[linecolor=red](q6)(q7)
\drawwhitedot{q1}\drawwhitedot{q2}\drawwhitedot{q3}\drawwhitedot{q4}
\drawwhitedot{q5}\drawwhitedot{q6}\drawwhitedot{q7}\drawwhitedot{q8}
\drawwhitedot{q9}\drawwhitedot{q10}\drawwhitedot{q11}\drawwhitedot{q12}
\psset{unit=1.4mm}
\pnode(10;3){q1}\pnode(10;2){q2}\pnode(10;1){q3}\pnode(10;0){q4}
\pnode(10;11){q5}\pnode(10;10){q6}\pnode(10;9){q7}\pnode(10;8){q8}
\pnode(10;7){q9}\pnode(10;6){q10}\pnode(10;5){q11}\pnode(10;4){q12}
\rput(q1){$\scriptstyle1$}\rput(q2){$\scriptstyle2$}\rput(q3){$\scriptstyle3$}\rput(q4){$\scriptstyle4$}\rput(q5){$\scriptstyle5$}\rput(q6){$\scriptstyle6$}\rput(q7){$\scriptstyle7$}\rput(q8){$\scriptstyle8$}\rput(q9){$\scriptstyle9$}\rput(q10){$\scriptstyle10$}\rput(q11){$\scriptstyle11$}\rput(q12){$\scriptstyle12$}
\end{picture}
\caption[]{\sf\smaller The polygons associated with a partition~$\pi$ and with its image under the complement map~$\partial$ and~$\partial^2$: the red partition is the Kreweras complement of~$\pi$, and repeating the operation leads to the image of the initial partition under a rotation by~$2\pi/\nn$, corresponding to conjugating under~$\Deltad_\nn$ in~$\BKL\nn$.}
\label{F:Kreweras}
\end{figure}
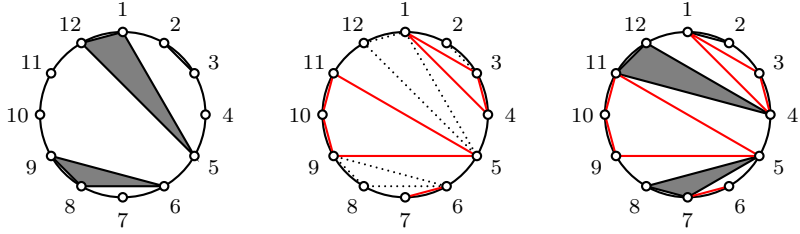

Putting pieces together, we obtain:

\begin{prop}
\label{P:Dual3}
For all~$\nn$ and~$\dd$, the number~$\bbs\nn\dd$ of~$\nn$-strand braids that have length at most~$\dd$ with respect to the Garside structure~$\SS_\nn^*$ is given by
\begin{equation}
\label{E:Dual3}
\bbs\nn\dd = \card\{(\pi_1 \wdots \pi_\dd) \in \NC\nn^\dd \mid \forall\ii< \nn \, (\overline{\pi_\ii} \wedge \pi_{\ii+1} = {\bf0}_\nn)\}.
\end{equation}
\end{prop}

\begin{proof}
By Lemma~\ref{L:Complement}\ITEM1, saying that $1$ is the only common left-divisor of~$\aa\pi$ and~$\aa{\pi'}$ in~$\BKL\nn$ amounts to saying that the convex hulls of the polygons associated with~$\pi$ and~$\pi'$ have no chord in common, hence are disjoint, so, in other words, that the meet of~$\pi$ and~$\pi'$ in the lattice~$\NC\nn$ is the minimal partition~${\bf0}_\nn$. By Lemma~\ref{L:Complement}\ITEM2, this condition has to be applied to $\overline{\pi_\ii}$ and~$\pi_{\ii+1}$ for every~$\ii$.
\end{proof}

\section{Free cumulants}
\label{S:Cumulants}

In view of Proposition~\ref{P:Dual3}, we have to count sequences of noncrossing partitions satisfying lattice constraints involving adjacent entries. We shall derive partial results from a general formula expressing the free cumulants of a product of independent random variables.

\subsection*{Noncrossing partitions and free cumulants} 

We recall some basic facts, more details can be found in~\cite{NiSp}.

Given a sequence of indeterminates $T_1 \wdots T_l,...$ and a noncrossing partition $\pi$ in~$\NC\nn$, one defines 
\begin{equation}
T_\pi=\prod_{p\in \pi} T_{|p|},
\end{equation}
where the product ranges over all parts $p$ of $\pi$ and $|p|$ is the number of elements of~$p$ in~$\{1 \wdots n\}$.

Two sequences of indeterminates $M_1 \wdots M_l,...$ and $R_1 \wdots R_l,...$
are related by the moment-cumulant formula if, for every $n$, one has
\begin{equation}
\label{MR}
M_n=\sum_{\pi\in \NC\nn}R_\pi.
\end{equation}
It is easy to see that this equation can be inverted and the $R_l$ can be expressed as polynomials in the $M_l$. In fact, introducing the generating series 
$$M(z)=1+\sum_{l=1}^\infty z^lM_l,\qquad R(z)=1+\sum_{l=1}^\infty z^lR_l,$$
we can recast the relation~\eqref{MR} in the form
\begin{equation}
\label{MRbis}
R(zM(z))=M(z).
\end{equation}
If the $M_l$ are the moments of a probability measure (or a random variable), then the quantities~$R_l$ are called the \emph{free cumulants} of the probability measure (or random variable).

\subsection*{Free cumulants of a product of independent random variables}

We establish now the general formula for free cumulants of products of independent random variables stated as Theorem~\ref{T:CumProd} in the introduction. 

\begin{proof}[Proof of Theorem~\ref{T:CumProd}]
Let $M^{(i)}_n$ be the moments of $X_i$. Let us write the moment-cumulant formula for each of the variables in the product $X_1X_2\pdots X_k$. As the $n$th moment of  $X_1X_2\pdots X_k$ is
$$M_n=M_n^{(1)}\pdots M_n^{(k)},$$ 
one has
$$M_n=M_n^{(1)}\pdots M_n^{(k)}=\sum_{\pi_1 \wdots \pi_k\in \NC\nn}\prod_{i=1}^kR^{(i)}_{\pi_i}.$$
Let us now decompose the sum in the right hand side according to the value of 
$\pi=\pi_1\vee\pdots\vee\pi_k$. Since $\pi_i\le\pi$ holds, each part of $\pi_i$ is included in some part of $\pi$. Let $p$ be a part of $\pi$. The intersections $\pi_{i,p}:=\pi_i\cap p$  form a partition of the set $p$. If we identify $p$ with 
$\{1 \wdots |p|\}$ by the only increasing bijection, then the sets~$\pi_{i,p}$ 
form  noncrossing partitions of $\{1 \wdots |p|\}$. Furthermore one has
$$\pi_{1,p}\vee\pdots\vee\pi_{n,p}={\bf 1}_{|p|},$$
whence
$$M_n=\sum_{\pi\in \NC\nn}\prod_{p\in \pi}\left[\sum_{\pi_{1,p}\vee\pdots\vee\pi_{n,p}={\bf 1}_{|p|}}\prod_iR^{(i)}_{\pi_{i,p}}\right].$$
Defining the sequence $Q_n$ by 
$$Q_n=\sum_{\pi_1\vee\pdots\vee\pi_n={\bf 1}_n}\prod_iR^{(i)}_{\pi_i},$$
we obtain 
$$M_n=\sum_{\pi\in \NC\nn}\prod_{p\in \pi}Q_{|p|}=\sum_{\pi\in \NC\nn}Q_{\pi}.$$
Then it follows from~\eqref{MR} that the quantities~$Q_n$ are the free cumulants of the sequence of moments $M_1 \wdots M_n, ...$\,. 
\end{proof}

The above argument is closely related to an argument in \cite{Spe}. The first author would like to thank Roland Speicher for pointing out this reference.

\begin{coro}
\label{C:Cumulants}
The number of $k$-tuples $(\pi_1 \wdots \pi_k)$ in $\NC\nn^\kk$ satisfying $\pi_1\vee\pdots\vee\pi_k={\bf 1}_n$ is the $n$th free cumulant of the variable $X_1^2\pdots X_k^2$ where $X_1 \wdots X_n, ...$ are independent centered semi-circular variables of variance 1.
\end{coro}

\subsection*{Classical and Boolean cumulants}
Cumulants can also be defined using the lattice of all set partitions of $\{1 \wdots n\}$ (this is the classical case, studied by Rota, Sch\"utzenberger...), or the lattice of interval partitions (these are the Boolean cumulants, see \cite{SpeW}). In both cases it is immediate to check that the proof of Theorem~\ref{T:CumProd} goes through and gives a formula for computing the corresponding cumulants of a product.

\section{Back to braids}
\label{S:Back}

We now apply the result of Section~\ref{S:Cumulants} to braids.

\subsection*{Incidence matrices}

For every~$\nn$, there exists a binary relation on~$\NC\nn$ that encodes $\SS_\nn^*$-normality. We introduce the associated incidence matrix.

\begin{defi}
For $\nn \ge 2$, we let $\Mat\nn$ be the $\Cat\nn \times \Cat\nn$-matrix whose entries are indexed by noncrossing partitions, and such that $(\Mat\nn)_{\pi, \pi'}$ is $1$ (\resp $0$) if $\overline\pi \wedge \pi' = {\bf0}_\nn$ holds (\resp fails).
\end{defi}

For instance, if the partitions of~$\{1, 2, 3\}$ are enumerated in the order $\Dv, \Da, \Db, \Dc, \DD$, the matrix~$\Mat3$ is $\left(\begin{matrix}
1&0&0&0&0\\
1&1&0&1&0\\
1&1&1&0&0\\
1&0&1&1&0\\
1&1&1&1&1
\end{matrix}\right)$. It follows from the properties of~${\bf0}_\nn$ and~${\bf1}_\nn$ that the column of~${\bf0}_\nn$ and the row of~${\bf1}_\nn$ in~$\Mat\nn$ contain only ones, whereas the row of~${\bf0}_\nn$, its ${\bf0}_\nn$-entry excepted, and the column of~${\bf1}_\nn$, its ${\bf1}_\nn$-entry excepted,  contain only zeroes.

\begin{prop}
\label{P:Dual4}
For all~$\nn$ and~$\dd \ge 1$, the number~$\bbs\nn\dd$ is the sum of all entries in the matrix~$(\Mat\nn)^{\dd-1}$; in particular, $\bbs\nn2$ is the number of positive entries in~$\Mat\nn$.
\end{prop}

\begin{proof}
For~$\pi$ in~$\NC\nn$ let $\bbs\nn\dd(\pi)$ be the number of $\SS_\nn^*$-normal sequences of  length~$\dd$ whose last entry is~$\pi$. By Proposition~\ref{E:Dual3}, a length~$\dd$ sequence $(\pi_1 \wdots \pi_{\dd-1}, \pi)$ contributes to~$\bbs\nn\dd(\pi)$ if and only if  $(\pi_1 \wdots \pi_{\dd-1})$ contributes to~$\bbs\nn{\dd-1}(\pi_{\dd-1})$ and $(\pi_{\dd-1}, \pi)$ contributes to~$\bbs\nn2(\pi)$, that is, $(\Mat\nn)_{\pi_{\dd-1}, \pi} = 1$ holds. We deduce
$$\bbs\nn\dd(\pi) = \sum_{\pi' \in \NC\nn} \bbs\nn{\dd-1}(\pi') \cdot (\Mat\nn)_{\pi', \pi}.$$
From there, an obvious induction shows that $\bbs\nn\dd(\pi)$ is the $\pi$th entry in $(1, 1 \wdots  1) \, (\Mat\nn)^{\dd-1}$, and the result follows by summing over all~$\pi$.
\end{proof}

(Note that, for $\dd = 1$, Proposition~\ref{P:Dual4} gives $\bbs\nn1 = \Cat\nn$, which is indeed the sum of all entries in the size~$\Cat\nn$ identity-matrix.)

Using Theorem~\ref{T:CumProd}, we can now complete the proof of Proposition~\ref{P:Dual}.

\begin{proof}[Proof of Proposition~\ref{P:Dual}]
By Proposition~\ref{P:Dual4}, $\bbs\nn2$ is the number of positive entries in the matrix~$\Mat\nn$, that is, the number of pairs $(\pi, \pi')$ in~$\NC\nn^2$ satisfying $\overline\pi \wedge \pi' = {\bf0}_\nn$. As the Kreweras complement is bijective, this number is also the number of pairs $(\pi, \pi')$ in~$\NC\nn^2$ satisfying $\pi \wedge \pi' = {\bf0}_\nn$. By complementation, the latter is also the number of pairs satisfying $\pi \vee \pi' = {\bf1}_\nn$. By Corollary~\ref{C:Cumulants}, this number is the $\nn$th free cumulant of~$X_1^2X_2^2$ where $X_1,X_2$ are independent centered semi-circular variables of variance 1. The moments of the latter are the squares of the Catalan numbers, so \eqref{MRbis} gives the expected result.
\end{proof}

\subsection*{Further questions}

It is easy to compute the numbers~$\bbs\nn\dd$ for small values of~$\nn$ and~$\dd$, see Table~\ref{T:Values}.

\begin{table}[htb]
\begin{tabular}{c|r|r|r|r|r|r|r}
\VR(3,2)$\dd$&1&2&3&4&5&6&7\\
\hline
\VR(4,1)$\bbs1\dd$& 1 & 1 & 1 & 1 & 1&1&1\\
\VR(3,1)$\bbs2\dd$& 2 & 3 & 4 & 5 & 6&7&8\\
\VR(3,1)$\bbs3\dd$&5 &15 &83 &177 &367 &749 &1\,515\\
\VR(3,1)$\bbs4\dd$&14 &99 &556 &2\,856&14\,122 &68\,927 &334\,632\\
\VR(3,1)$\bbs5\dd$&42 &773 &11\,124 &147\,855 &1\,917\,046 &24\,672\,817\\
\VR(3,1)$\bbs6\dd$&132 &6\,743 &266\,944 &9\,845\,829 &356\,470\,124 &
\end{tabular}
\vspace{2mm}
\caption[]{\sf\smaller The number~$\bbs\nn\dd$ of $\nn$-strand braids of length at most~$\dd$ in the dual braid monoid~$\BKL\nn$: the first column ($\dd = 1$) contains the Catalan numbers, whereas the second column contains the sequence specified in Proposition~\ref{P:Dual}, which is A168344 in~\cite{Slo}.}
\label{T:Values}
\end{table}

A natural question is to ask for a description of the columns in Table~\ref{T:Values} beyond the first two ones. The characterization of Proposition~\ref{P:Dual} does not extend to $\dd \ge 3$: for instance, we have
$$\bbs\nn3 = \{(\pi_1, \pi_2, \pi_3) \in \NC\nn^3 \mid \overline{\pi_1} \wedge \pi_2 = {\bf0}_\nn \mbox{ and } \overline{\pi_2} \wedge \pi_3 = {\bf0}_\nn\}:$$
replacing $\wedge$ and ${\bf0}_\nn$ with $\vee$ and ${\bf1}_\nn$ is easy, but the Kreweras complement cannot be forgotten in this case.

On the other hand, describing the rows in Table~\ref{T:Values} leads to further natural questions. By Proposition~\ref{P:Dual4}, the generating series of the numbers~$\bbs\nn\dd$ is rational for every~$\nn$, and $\bbs\nn\dd$ can be expressed in terms of the $\dd$-power of the eigenvalues of the matrix~$\Mat\nn$. For instance, one easily finds $\bbs3\dd = 6 \cdot 2 ^\dd - 2\dd - 5$ for every~$\dd$, as well as $\bbs3\dd(\pi) = 2^{\dd + 1} - 1$ for $\pi \not= {\bf0}_\nn, {\bf1}_\nn$ (as above, we write $\bbs\nn\dd(\pi)$ for the number of braids with a normal form finishing with~$\pi$). Very little is known for $\nn \ge 4$.

It would be of interest to compute, or at at least approximate, the spectral radius of the matrix~$\Mat\nn$. Here are the first values.
\begin{table}[htb]
\begin{tabular}{c|c|c|c|c|c|c|c}
\VR(3,2)$n$&$1$&$2$&$3$&$4$&$5$&$6$&$7$\\
\hline
\VR(4,1)$\rho(\Mat\nn)$&$1$&$1$&$2$&$4.83...$&$12.83...$&$35.98...$&$104.87...$
\end{tabular}
\medskip
\caption[]{\sf\smaller Spectral radius of the incidence matrix~$\Mat\nn$}
\label{T:Radius}
\end{table}

\subsection*{A determinant}
Although we are not able to compute explicitly the eigenvalues of the matrix~$\Mat\nn$, we can give a closed formula for its determinant.

\begin{theorem}
\label{T:Det}
 For every~$\nn$, we have 
$$|\det(\Mat\nn)|=\prod_{k=2}^{n}\Cat\kkl^{2n-k-1\choose n-1}.$$
\end{theorem}

Before proving this formula, we need an auxiliary result  involving M\"obius matrices. By definition, if  $X$  is  a finite poset,  the associated M\"obius matrix~$\mu$ is the inverse of  the order matrix~$\zeta$ indexed by the elements of $X$  and given by 
$$\zeta(x,y)=\begin{cases}1&\text{for $x\geq y$},\\0&\text{otherwise}.\end{cases}$$
If the elements of $X$ are ordered according to a linear extension of the partial order of $X$, then $\zeta$ is a lower triangular matrix, with  ones  on the diagonal.  In  particular, its determinant is 1. The same is true for $\mu$.

\begin{lemm}
\label{L:det}
Assume that $X$ is a lattice,  and $\varphi$ is a complex valued function on $X$. Let $\Phi$ be the matrix  defined by  $\Phi(x,y)=\varphi(x\wedge y)$.  Then we have 
$$\det(\Phi)=\prod_{x\in X}\hat\varphi(x),$$
 where $\hat\varphi$ is given by  $\hat \varphi(x)=\sum_{y\leq x}\mu(x,y)\varphi(y)$.
\end{lemm}

 Note that, under the above assumptions, we have $\varphi(x)=\sum_{y\leq x}\hat \varphi(y)$ by definition of the M\"obius matrix. 

\begin{proof}[Proof of Lemma~\ref{L:det}]
First observe that $\zeta(x\wedge y,z)=\zeta(x,z)\zeta( y,z)$, therefore
\begin{eqnarray*}\Phi(x,y)=\varphi(x\wedge y)&=&\sum_{z\in X}\zeta(x\wedge y,z)\hat \varphi(z)
\\&=&\sum_{z\in X}\zeta(x,z) \zeta(y,z)\hat \varphi(z).
\end{eqnarray*}
 Multiplying  the matrix $\Phi$ on the left by $\mu$, one obtains
\begin{eqnarray*}\mu*\Phi(x,y)&=&\sum_u\mu(x,u)\Phi(u,y)\\&=&\sum_u\sum_z\mu(x,u)\zeta(u,z) \zeta(y,z)\hat \varphi(z)\\&=&\zeta(y,x)\hat \varphi(x),\end{eqnarray*}
 the last equality coming from  $\sum_u\mu(x,u)\zeta(u,z)=\delta_{xz}$.
It follows that the matrix $\mu*\Phi$ is upper diagonal, with the numbers $\hat \varphi(x)$ on the diagonal.
\end{proof}

\begin{proof}[Proof of Theorem~\ref{T:Det}]
We apply  Lemma~\ref{L:det}  to $X=\NC\nn$  and $\varphi$ defined by  $\varphi(\pi)=1$ for $\pi={\bf0}_n$ and $\varphi(\pi)=0$ for $\pi\ne {\bf0}_n$. We obtain $\hat\varphi(\pi)=\mu(\pi,{\bf0}_n)$. It is known that the  M\"obius function for the lattice of noncrossing partitions is multiplicative: if $\pi$ is made of $p_k$ parts of size $k$ for $k=1,2,...$, then  one has 
$$\mu(\pi,{\bf0}_n)=\prod_k((-1)^{k-1}\Cat\kkl)^{p_k}.$$ 

The matrix $\Mat\nn$ coincides with the  corresponding  matrix $\Phi$ up to a permutation of the columns (given by taking the Kreweras complement).
 We deduce 
$$|\det(\Mat\nn)|=\prod_k\Cat\kkl^{a_{k,n}},$$
 with  $$a_{k,n}=\sum_{\pi\in \NC\nn}p_k(\pi).$$
 In order to compute $a_{k,n}$, we introduce the generating functions
$$f_{n,k}(y)=\sum_{\pi\in\NC\nn}y^{p_k} \quad\text{ and } \quad f(z,y,k)=\sum_{n\geq k}z^nf_{n,k}(y).$$ 
 Then  one has
$$a_{k,n}=f_{n,k}'(1),$$ whereas $f_n(1)=\Cat\nn$ and
$$\phi(z):=f(z,1,k)=\frac{1-\sqrt{1-4z}}{2z}.$$ 
Let $\pi$ be  a noncrossing partition, such that the size of the part containing~1 is~$l$. The complement of this part  consists  of $l$ (possibly empty) noncrossing partitions $\pi_1 \wdots \pi_l$, whose sizes add up to $n-l$. Moreover,  we find 
\begin{align*}
p_k(\pi)
&=p_k(\pi_1)+\pdots+p_k(\pi_l)
&\text{ for  $l\ne k$},\\
&=p_k(\pi_1)+\pdots+p_k(\pi_l)+1
&\text{ for  $l=k$}.
\end{align*}
It follows from these considerations that the generating function  $f(z,y,k)$ satisfies the equation
$$f(z,y,k)=1/(1-zf(z,y,k))+(y-1)z^kf^k(z,y,k).$$
In particular, for $y=1$, we recover 
\begin{equation}\label{phi}\phi(z)=1/(1-z\phi(z)).\end{equation}
Differentiating with respect to $y$ and evaluating at $y=1$, we  find 
$$f_y(z,1,k)=zf_y(z,1,k)/(1-z\phi(z))^2+z^k\phi^k(z).$$
By (\ref{phi}), one has
$$1-z/(1-z\phi(z))^2=2-\phi(z),$$  whence 
$$\sum_{n\geq k}a_{n,k}z^n=f_y(z,1,k)=z^k\phi^k(z)/(2-\phi(z)).$$
It follows that the  generating series $h_k:=f_y(z,1,k)$ satisfy the  inductive rules  $h_{k+1}=h_k-zh_{k-1}$.  By  Pascal's relation for binomial coefficients,
 the series  $g_k:=\sum_{n\geq k}z^n{2n-k-1\choose n-1}$ satisfy the same  induction rules.  It remains to check $g_k=h_k$ for $k=0$  and  $1$, which is an easy exercise left to the reader,  to conclude that $g_k = h_k$ holds for every~$k$. 
\end{proof}

\end{document}